\documentclass[a4paper,11pt]{amsart}
\usepackage[english]{babel} %
\usepackage{mathtools} %
\usepackage{amssymb}
\usepackage[T1]{fontenc}
\usepackage{lmodern}
\usepackage{mathrsfs} %
\usepackage{slashed} %
\usepackage[babel=true]{microtype} %
\usepackage[autostyle=true]{csquotes} %
\usepackage[dvipsnames]{xcolor} %
\calclayout
\usepackage[
    citestyle=numeric-comp,
    bibstyle=numeric-comp,
    backend=biber,
    url=true,
    doi=true,
    isbn=false,
    giveninits=true,
    maxbibnames=100,
    maxcitenames=10,
    sorting=nyt
]{biblatex}

\DeclareDelimFormat[textcite]{multinamedelim}{\textendash}
\DeclareDelimAlias[textcite]{finalnamedelim}{multinamedelim}

\DeclareSourcemap{
  \maps{
    \map{ %
        \step[fieldset = month, null]
        \step[fieldset = day, null]
        \step[fieldset = language, null]
    }
    \map{ %
        \step[fieldsource = doi, final]
        \step[fieldset = url, null]
    }
    \map{ %
        \step[fieldsource = eprint, final]
        \step[fieldset = url, null]
    }
  }
}

\usepackage[biblatex=true]{embrac} %

\usepackage{tikz}
\usetikzlibrary{cd} %

\usepackage[shortlabels]{enumitem}
\setlist[enumerate,1]{label=\upshape(\arabic*)}
\newlist{myenumi}{enumerate}{1}
\setlist[myenumi,1]{label=\upshape(\roman*)}
\newlist{myenuma}{enumerate}{1}
\setlist[myenuma,1]{label=\upshape(\alph*)}

\numberwithin{equation}{section}
\allowdisplaybreaks[1]

\newtheorem{theorem}{Theorem}[section]
\newtheorem*{theorem*}{Theorem}
\usepackage{thmtools}

\definecolor{Heather}{RGB}{164, 132, 172}
\usepackage[pdfusetitle,bookmarks,pdfpagelabels]{hyperref} %
\hypersetup{
  colorlinks,
  urlcolor=Periwinkle,
  citecolor=Heather,
  linkcolor=teal,
  breaklinks=true,
  final,
}

\declaretheorem[name=Lemma, numberlike=theorem]{lemma}
\declaretheorem[name=Lemma, numbered=no]{lemma*}

\declaretheorem[name=Corollary, numbered=no]{corollary*}
\declaretheorem[name=Proposition, numberlike=theorem]{proposition}
\declaretheorem[name=Definition, numberlike=theorem, style=definition]{definition}

\declaretheorem[name=Remark, numberlike=theorem, style=remark]{remark}

\declaretheorem[name=Theorem]{theoremx}
\declaretheorem[name=Corollary, numberlike=theoremx]{corollaryx}

\usepackage[noabbrev,capitalise]{cleveref} %
\crefdefaultlabelformat{#2\textup{#1}#3}
\crefname{theoremx}{Theorem}{Theorems}
\Crefname{theoremx}{Theorem}{Theorems}
\crefname{corollaryx}{Corollary}{Corollaries}
\Crefname{corollaryx}{Corollary}{Corollaries}

\usepackage{todonotes}

\makeatletter
\providecommand\@dotsep{5}
\def\listtodoname{List of Todos}
\def\listoftodos{\@starttoc{tdo}\listtodoname}
\makeatother
\NewDocumentEnvironment{mytodoenv}{m}{\hypersetup{hidelinks}\textsf{\textbf{#1:}} }{}

\usepackage{mymacros}

\title[Scalar curvature rigidity for products of convex hypersurfaces]{Scalar curvature rigidity for products of convex hypersurfaces}
\author{Samuel Lockman}
\thanks{Funded by the Deutsche Forschungsgemeinschaft (DFG, German Research Foundation) – Project numbers 224262486; %
313840899 (Lockman)  %
– 
523079177; %
427320536; %
390685587 (Zeidler)} %

\address[Lockman]{Universität Regensburg, Fakultät für Mathematik, 93040 Regensburg, Germany}
\email{\href{mailto:samuel.lockman@mathematik.uni-regensburg.de}{samuel.lockman@mathematik.uni-regensburg.de}}
\author{Rudolf Zeidler}
\thanks{(Zeidler) Funded by the European Union (ERC Starting Grant 101116001 – COMSCAL). Views and opinions expressed are however those of the author(s) only and do not necessarily reflect those of the European Union or the European Research Council. Neither the European Union nor the granting authority can be held responsible for them.}
\address[Zeidler]{Universität Potsdam, Institut für Mathematik, Karl-Liebknecht-Str.\ 24--25, 14476 Potsdam, Germany}
\email{\href{mailto:rudolf.zeidler@uni-potsdam.de}{rudolf.zeidler@uni-potsdam.de}}
\urladdr{\href{https://www.rzeidler.eu}{www.rzeidler.eu}}

\hypersetup{
  pdfauthor={Samuel Lockman, Rudolf Zeidler}, %
}
\begin{document}

\begin{abstract}
    Let $N = N_{1} \times \dotsm \times N_{k}$, where each $N_{i} \subset \mathbb{R}^{n_i+1}$ is a closed strictly convex hypersurface. Let $M$ be a Riemannian spin manifold of dimension $n = \dim(N)$, and let $f \colon M \to N$ be an area non-increasing smooth map of non-zero degree.
    We show that $\mathrm{scal}_M \geq \mathrm{scal}_N \circ f$ implies $\mathrm{scal}_M = \mathrm{scal}_N \circ f$. Moreover, if $n \geq 3$ and $N$ has no circle factors, then every such map is a Riemannian isometry.
    In the presence of circle factors, we obtain the corresponding optimal splitting theorem for $M$ and $f$.
    Our results are based on an approach to the index-theoretic part of Llarull's scalar curvature rigidity theorem via Clifford-linear family index theory, which works independently of the parity of the dimension and extends naturally to products.
    This includes a proof of the Geroch conjecture for spin manifolds as the edge case with only circle factors.
\end{abstract}
\maketitle

\section{Introduction}
Llarull's classical rigidity theorem~\cite{Llarull} says that a smooth area non-increasing map \(f \colon M \to \Sphere^n\) of non-zero degree, where \(M\) is a closed connected Riemannian spin manifold of dimension \(n \geq 3\) satisfying \(\scal_{M} \geq n(n-1)\), is an isometry.

In this article, we are interested in generalizations of Llarull's theorem obtained by replacing the target manifold \(\Sphere^n\) with other closed manifolds.
Indeed, \textcite{Goette-Semmelmann} extended Llarull's proof to the following:
Let \((N,g_N)\) be a smooth Riemannian manifold of dimension \(n \geq 3\) with non-negative curvature operator, satisfying \(\tfrac{\scal_{g_N}}{2} g_{N}> \Ric_{g_N} > 0\), and with non-zero Euler characteristic \(\chi(N) \neq 0\). Then any area non-increasing smooth spin map \(f \colon M \to N\) of non-zero degree with \(\scal_M \geq \scal_{g_N} \circ f\) must be an isometry.
The proof of these results is based on carefully analyzing the Schrödinger--Lichnerowicz formula for the Dirac operator on \(M\) twisted by the pullback of the spinor bundle of the target manifold \(N\).
As a second ingredient, the argument needs the existence of non-trivial harmonic sections and here index theory enters. The ostensibly natural index problem in this context leads to \(\deg(f) \chi(N)\) as the relevant invariant which is the reason for the Euler characteristic hypothesis.
As it stands, this argument would not cover the full range of Llarull's original result because it would exclude odd-dimensional spheres.
Instead, Llarull approaches this via an argument involving products with large circles in order to reduce the case of odd-dimensional spheres to the even-dimensional case.
Circle suspension tricks are a common pattern in the application of index theory to scalar curvature problems in order to treat odd-dimensional cases, where there is no classical index.
However, there is usually a more conceptual approach.

\textcite{SpectralFlowLlarull} found a natural argument to prove Llarull's theorem for odd-dimensional spheres via the spectral flow of the Dirac operator twisted by a suitable \(1\)-parameter family of connections; see also~\textcite{baer2025spectralflowatiyahpatodisingerindex}.
Indeed, the spectral flow is the appropriate classical device to extract topological information out of Dirac-type operators in odd dimensions.
In \(\K\)-theoretic language, twisting an operator by a bundle and taking the index implements the pairing \(\K_0(M) \times \K^{0}(M) \to \Z\) between K-homology and K-theory, whereas the spectral flow yields the odd counterpart \(\K_1(M) \times \K^{1}(M) \to \Z\).
This brings the even- and odd-dimensional cases of Llarull's theorem superficially on the same footing and it also treats arbitrary strictly convex hypersurfaces \(N \subset \R^{n+1}\).
However, in the concrete argument, the topological reason providing the existence of a non-trivial harmonic section is still fundamentally different between even- and odd dimensions.
This presents a challenge for further generalizations, for instance to a product like \(\Sphere^3 \times \Sphere^3\): although it is even-dimensional, the twisted Dirac operators that are geometrically relevant in the Llarull
setting admit no obvious non-vanishing classical index.

In this article, we systematically solve the problem for products of strictly convex hypersurfaces via a Clifford-linear multi-parameter version of the spectral flow proof.
That is, for a product \(N = N_1 \times \dotsm \times N_m\), we construct a geometrically useful family of Clifford-module bundles over \(N\) parametrized by the \(m\)-torus \(\Torus^m\) whose family index class lives in \(\K^{m}(\Torus^m)\) and equals the Bott class.
In \cref{sec:index-prelim}, we first describe this index-theoretic setup for a single factor and then exhibit a uniform proof of the classical Llarull theorem in \cref{sec:llarull-classical}.
In \cref{sec:products}, we prove our main theorem:

\begin{theoremx}\label{thm:general-product-splitting}
    Let $N_j \subset (\R^{n_j+1}, g_{\mathrm{eucl}})$ be closed and strictly convex hypersurfaces with \(n_j \geq 2\) for $j = 1, \dotsc, k$ and $\sum_j n_j \eqcolon n$.
    We define \(N = \prod_{j} N_j\) and let the induced product metric be denoted by \(g_N\).

    Let $(M,g)$ be a closed connected Riemannian spin manifold of dimension $n + l$ with \(l \geq 0\).
    Let \(T = \Sphere^1_{r_1} \times \dotsm \times \Sphere^{1}_{r_l}\) be a product of \(l\) circles and \(g_T\) be the corresponding flat product metric on \(T\).
    Suppose that $f \colon M \to N \times T$ is a smooth map of non-zero degree such that
    \begin{itemize}
        \item \(f \colon (M,g) \to (N \times T,g_N \oplus g_{T})\) is area non-increasing,
        \item $\scal_g \geq \scal_{g_N \oplus g_T} \circ f = \scal_{g_N} \circ \proj_N \circ f$
    \end{itemize}
     
    Then
    \begin{description}
        \item[if \(n \geq 3\)] There exists a flat torus \((F,g_F)\), a smooth \(1\)-Lipschitz map of non-zero degree \(h \colon (F,g_F) \to (T,g_T)\), and a Riemannian isometry \(\Phi \colon (M,g) \to (N \times F, g_N \oplus g_F)\) such that the diagram %
    \begin{equation*}
        \begin{tikzcd}
            M \ar[r, "\Phi", "\cong"'] \ar[rd, "f"'] & N \times F \ar[d, "\id_N \times h"] \\
            & N \times T
        \end{tikzcd}
    \end{equation*}
    commutes.
    \item[if \(n = 2\)] There exists a flat torus \((F,g_F)\), a smooth surface $(S,g_S)$, a smooth \(1\)-Lipschitz map of non-zero degree \(h \colon (F,g_F) \to (T,g_T)\), an area-density-preserving diffeomorphism $\phi \colon S \to N$ that preserves the Gauß curvatures in the sense that \(\Gauss_S = \Gauss_N \circ \hspace{0.85mm} \phi\), and a Riemannian isometry \(\Phi \colon (M,g) \to (S \times F, g_S \oplus g_F)\) such that the diagram %
    \begin{equation*}
        \begin{tikzcd}
            M \ar[r, "\Phi", "\cong"'] \ar[rd, "f"'] & S \times F \ar[d, "\phi \times h"] \\
            & N \times T
        \end{tikzcd}
    \end{equation*}
    commutes.
    
    \item[if \(k = n = 0\)] \((M,g)\) is isometric to a flat torus.
    \end{description}
\end{theoremx}

Note that the reason we treat circle factors separately in the statement is because they play a different role than the others for the rigidity conclusion.
For the index-theoretic part of the proof, they are in fact treated in the same way.
If there are no circle factors present, then the theorem simplifies to a Llarull-type statement:
\begin{corollaryx}\label{thm:llarull-product}
    Let $N_j \subset (\R^{n_j+1}, g_{\mathrm{eucl}})$ be closed and strictly convex hypersurfaces with \(n_j \geq 2\) for $j = 1, \dotsc, k$ and suppose that $\sum_j n_j \eqcolon n \geq 3$.
    Let $(M,g)$ be a closed connected Riemannian spin manifold of dimension $n$.
    Then any area non-increasing smooth map \(f \colon M \to N \coloneq \prod_{j} N_j\) of non-zero degree with $\scal_g \geq \scal_{g_N} \circ f$ is an isometry.
\end{corollaryx}

We note that the prior literature already covers many cases of \cref{thm:general-product-splitting,thm:llarull-product}.
If all factors are even-dimensional, \cref{thm:llarull-product} is contained in the \textcite{Goette-Semmelmann} result, because then \(\chi(N) = 2^k \neq 0\) and the condition \(\tfrac{\scal_{g_N}}{2} g_{N}> \Ric_{g_N} > 0\) is automatically satisfied due to the strict convexity assumption for each factor if the total dimension is \(\geq 3\).
Similarly, if we have a product of even-dimensional factors with some one-dimensional factors, then the corresponding case of \cref{thm:general-product-splitting} is covered by \textcite{riedler-tony} as a consequence of their more sophisticated product splitting result for general scalar-rigid submersions with non-vanishing higher-index degree (see in particular \cite[Example C]{riedler-tony}).
As we already discussed, the case of a single odd-dimensional factor in \cref{thm:llarull-product} was treated in \cite{SpectralFlowLlarull,baer2025spectralflowatiyahpatodisingerindex} via the classical spectral flow.
Moreover, by combining minimal slicing techniques with the classical Llarull theorem, cases of the form \(N = \Sphere^{n} \times \mathrm{T}^l\) with \(n + l \leq 7\) were treated by \textcite{hao-shi-sun:llarull-slicing,chow}.

In the edge case \(n = 0\), the condition that the map be area non-increasing can be dropped without loss of generality, so we recover the classical result of \textcite{schoen-yau} and \textcite{Gromov-Lawson} on the Geroch conjecture in the spin setting:
\begin{corollary*}
    Let \(M\) be a closed connected spin manifold which admits a map of non-zero degree to the torus of the same dimension.
    If \(g\) is a Riemannian metric on \(M\) with non-negative scalar curvature, then \((M,g)\) is a flat torus.
\end{corollary*}
We note that the approach to positive scalar curvature obstructions via the Rosenberg index~\cite{Rosenberg:PSCNovikovI} also reduces to a family index in the case of the torus.
However, the family index we use here is a slightly different one; see \cref{rem:family-index-comparison}.

In sum, we provide a uniform index-theoretic framework to prove Llarull- and Geroch-type theorems for products of strictly convex hypersurfaces, including products of arbitrary many round spheres and tori.
This framework both recovers the existing results and extends to the missing cases of multiple factors of dimension at least \(2\) which are not all even-dimensional.

\section{Index-theoretic preliminaries} \label{sec:index-prelim}

We start with preliminaries concerning Clifford algebras, family index theory, K-homology and K-theory.
All of this can be conveniently phrased in terms of Kasparov's bivariant K-theory for \(\Z/2\)-graded \textCstar-algebras \cite{Kasparov:EquivariantKK}, which is the point of view we will take here.

Let \(\CCl_{p,q}\) denote the complex Clifford algebra on \(\R^{p,q} \coloneq \R^p \oplus \R^q\) with the relation \((v \oplus w)^2 = (\abs{w}^2 - \abs{v}^2) 1\) for \(v\oplus w \in \R^{p,q}\).
Given an \(m\)-dimensional Riemannian spin manifold \(M\), we will consider the \(\CCl_{m,0}\)-linear Dirac operator acting on the bundle \(\SpinBdl_M = \mathrm{P}_{\Spin(m)}(M) \times_{\Spin(m)} \CCl_{m,0}\), where \(\Spin(m) \subset \CCl_{m,0}\) acts via left multiplication.
This Dirac operator defines a fundamental class \([\Dirac_{M}] \in \K_{m}^\lf(M) = \KK(\Ct_0(M), \CCl_{m,0})\).
Dually, we have the Bott class \(\beta_M \in \K^{m}_{\cpt}(M)\) which satisfies \([\Dirac_M] \cap \beta_M = 1 \in \K^0(M)\) with respect to the cap product \(\cap\), which can be seen as the map making the diagram
\[
    \begin{tikzcd}
        \K_{m}^{\lf}(M) \otimes \K^{m}_{\cpt}(M) \ar[r, equal] \ar[dd, "\cap"] & \KK(\Ct_0(M), \CCl_{m,0}) \otimes \KK(\C, \Ct_0(M) \otimes \CCl_{0,m}) \ar[d, "\id \otimes \mathrm{diag}^\ast"]\\
        & \KK(\Ct_0(M), \CCl_{m,0}) \otimes \KK(\C, \Ct_0(M) \otimes \Ct(M) \otimes \CCl_{0,m}) \ar[d, "\otimes_{\Ct_0(M)}"] \\
        \K^{0}(M) \ar[r, equal] & \KK(\C, \Ct(M) \otimes \CCl_{m,m}).
    \end{tikzcd}
\]
commute. The Bott class on a connected manifold \(M\) has the property that \(\beta_{M} = \iota_{!} \beta_{B}\), where \(\iota \colon B \hookrightarrow M\) is any coordinate ball.
This implies that for a proper smooth map \(f \colon M' \to M\) between connected spin \(m\)-manifolds, we have
\begin{equation}
    f^\ast \beta_{M} = \deg(f) \beta_{M'} \in \K^{m}_{\cpt}(M'). \label{eq:bott-degree-formula}
\end{equation}
Moreover, both fundamental classes and Bott classes are compatible with exterior products, that is, \([\Dirac_{M_1 \times M_2}] = [\Dirac_{M_1}] \times [\Dirac_{M_2}]\) and 
\begin{equation}
    \beta_{M_{1} \times M_{2}} = \beta_{M_1} \times \beta_{M_2} \label{eq:bott-product}
\end{equation}
if \(M_1\) and \(M_2\) are both spin manifolds.

\begin{remark} \label{rem:bott_classes}
Now consider the case of a closed manifold \(M\) and note that a bundle \(E \to M\) of \(\Z/2\)-graded right \(\CCl_{p,q}\)-modules represents a class \([E] \in \K^{q-p}(M) = \KK(\C, \Ct(M) \otimes \CCl_{p,q})\).
The Bott class \(\beta_{M}\) can be explicitly described via a \(\CCl_{0,m}\)-bundle as follows.
Take a coordinate unit disk \(\Disk_{1}(0) \cong D \subset M\) and let \(E \to M\) be the bundle obtained by clutching the trivial \(\CCl_{0,m}\)-modules on \(D\) and \(M \setminus \interior{D}\), respectively, along \(\Sphere^{m-1} \cong \partial D\) via the map \(\gamma \colon \Sphere^{m-1} \to \Aut^{(0)}_{\CCl_{0,m}}(\CCl_{0,m})\), \(x \mapsto (a \mapsto e_{m} x a)\), where we view \(\CCl_{0,m}\) as a \(\Z/2\)-graded right module over itself.
The role of the fixed coordinate vector \(e_{m}\) in the definition of \(\gamma\) is to implicitly identify the opposite module grading on \(\CCl_{0,m}\) with the standard one.
\end{remark}

Next we turn to family index theory.
Consider a submersion \(p \colon Y \to X\) between closed spin manifolds, where \(\dim(Y) = n +k\) and \(\dim(X) = k\).
The family of vertical \(\CCl_{n,0}\)-linear Dirac operators on the fibers of \(p\) defines a fundamental class \([\Dirac_{p}] \in \KK(\Ct(Y), \Ct(X) \otimes~\CCl_{n,0})\) which allows us to decompose the fundamental class \([\Dirac_{Y}] \in \K_{n+k}(Y)\) as
\[
[\Dirac_{Y}] = [\Dirac_{p}] \otimes_{\Ct(X)} [\Dirac_X].
\]
Now given a bundle \(E \to Y\) of \(\Z/2\)-graded right \(\CCl_{0,n+k}\)-modules together with a family of \(\CCl_{0,n+k}\)-linear vertical connections \((\nabla_x)_{x \in X}\) on \(E\), parametrized by $X$, we can twist the family of vertical Dirac operators with \((E, \nabla_\bullet)\) and obtain a family of Dirac-type operators \(\left(\Dirac_{p^{-1}(x), E|_{p^{-1}(x)}}\right)_{x \in X}\) which has a family index
\begin{gather*}
\ind_{X} \left(\Dirac_{p^{-1}(\bullet), E|_{p^{-1}(\bullet)}}\right) = p_{!}([E]) \coloneq [E] \otimes_{\Ct(Y)} [\Dirac_p] \\
\in \KK(\C, \Ct(X)\otimes \CCl_{n, n+k}) = \K^k(X).
\end{gather*}
Concretely, this is the index class of the twisted vertical Dirac operator acting as an unbounded \(\Ct(X) \otimes \CCl_{n,n+k}\)-linear operator on \(\Lp^2_{\Ct(X)}(\SpinBdl_{Y, \mathrm{vert}} \otimes E)\), which is the completion of \(\Ct^\infty(Y; \SpinBdl_{Y, \mathrm{vert}} \otimes E)\) subject to the \(\Ct(X, \CCl_{n,n+k})\)-valued fiberwise-\(\Lp^2\) scalar product defined by integrating the pointwise \(\CCl_{n,n+k}\)-valued scalar product over the fibers.
From this description one concludes the following:
\begin{lemma}\label{lem:invertible_everywhere_vanishes}
    Suppose that \(\Dirac_{p^{-1}(x), E|_{p^{-1}(x)}}\) is invertible for every \(x \in X\).
    Then the family index class \(\ind_{X} \left(\Dirac_{p^{-1}(\bullet), E|_{p^{-1}(\bullet)}}\right) \in \K^{k}(X)\) vanishes.
\end{lemma}

Moreover, non-vanishing of the total twisted index class implies non-vanishing of the family index class:
\begin{lemma}\label{lem:total_nonvanish_implies_family-nonvanish}
    If \(\innp{[\Dirac_Y]}{[E]} \neq 0 \in \Z\), then \(\ind_{X} \left(\Dirac_{p^{-1}(\bullet), E|_{p^{-1}(\bullet)}}\right) \neq 0 \in \K^k(X)\).
\end{lemma}
\begin{proof}
Suppose by contraposition that \(0 = \ind_{X} \left(\Dirac_{p^{-1}(\bullet), E|_{p^{-1}(\bullet)}}\right) = [E] \otimes_{\Ct(Y)} [\Dirac_p]\).
Then
\begin{align*}
    \innp{[\Dirac_Y]}{[E]} &= [E] \otimes_{\Ct(Y)} [\Dirac_Y] \\
    &= [E] \otimes_{\Ct(Y)} \left( [\Dirac_p] \otimes_{\Ct(X)} [\Dirac_X] \right) \\
    &= \left( [E] \otimes_{\Ct(Y)} [\Dirac_p] \right) \otimes_{\Ct(X)} [\Dirac_X] \\
    &= 0 \otimes_{\Ct(X)} [\Dirac_X] = 0. \qedhere
\end{align*}
\end{proof}

We now apply the above discussion in the setting \(X = \Torus \coloneq [0,1]/(0\sim 1)\) and \(Y =  \Torus \times \Sphere^n\), \(p = \proj_1\).

\begin{definition}\label{defi:mapping_torus_bundle}
Let \(\Sphere^n \subset \R^{n+1}\) and let \(\gamma \colon \Sphere^n \to \Aut^{(0)}_{\CCl_{0,n+1}}(\CCl_{0,n+1})\) be the map \(x \mapsto (a \mapsto e_{n+1} x a)\), where we view \(\CCl_{0,n+1}\) as a \(\Z/2\)-graded right module over itself.
Then let \(\tilde{B}^{n+1} \to \Torus \times \Sphere^{n}\) denote the bundle of graded \(\CCl_{0,n+1}\)-modules obtained by gluing the endpoints of the trivial \(\CCl_{0,n+1}\)-bundle on \([0,1] \times \Sphere^n\) via \(\gamma\).
This can be viewed as a family of bundles \(\{\tilde{B}^{n+1}_{\bullet} \to \Sphere^n\}_{\bullet \in \Torus}\), which for convenience will be denoted by \(\tilde{B}^{n+1}_{\bullet} \to \Sphere^n\) or simply \(\tilde{B}^{n+1}_{\bullet}\).
\end{definition}
We twist the \(\CCl_{n,0}\)-linear Dirac operator on \(\Sphere^n\) with the family \(\tilde{B}^{n+1}_{\bullet}\) and denote the corresponding family of twisted Dirac operators by \(\Dirac_{\Sphere^{n}, \tilde{B}^{n+1}_{\bullet}}\).
We obtain the family index class \(\ind_{\Torus}(\Dirac_{\Sphere^{n}, \tilde{B}^{n+1}_{\bullet}}) = (\proj_1)_![\tilde{B}^{n+1}] \in \KK(\C, \Ct(\Torus) \otimes \CCl_{n,n+1}) = \K^1(\Torus)\).

\begin{lemma} \label{lem:family_index_main_computation}
    The bundle \(\tilde{B}^{n+1}\) represents the Bott class of \(\Torus \times \Sphere^n\), that is,
    \begin{equation} \label{eq:S^1timesS^n-bott-bundle}
        [\tilde{B}^{n+1}] = \beta_{\Torus \times \Sphere^n} \in \K^{n+1}(\Torus \times \Sphere^n).
    \end{equation}
    In particular, \(\innp{[\Dirac_{\Torus \times \Sphere^n}]}{[\tilde{B}^{n+1}]} = 1 \) and the family index class \(\ind_{\Torus}(\Dirac_{\Sphere^{n}, \tilde{B}^{n+1}_{\bullet}}) \in \K^1(\Torus)\) does not vanish for any choice of connection family \(\nabla_\bullet\) on \(\tilde{B}^{n+1}_{\bullet}\).
\end{lemma}
\begin{proof}
    Using \(\gamma\) as a clutching map also yields a bundle \(\BottBdl_{n+1} \to \Sphere^{n+1} = \Disk^{n+1} \cup_{\Sphere^n} \Disk^{n+1}\) of graded \(\CCl_{0,n+1}\)-modules on \(\Sphere^{n+1}\).
    This bundle represents the Bott class \(\beta_{\Sphere^{n+1}} = [\BottBdl_{n+1}] \in \K^{n+1}(\Sphere^{n+1}) = \KK(\C, \Ct(\Sphere^{n+1}) \otimes \CCl_{0,n+1})\) by \cref{rem:bott_classes}, and hence we will refer to this bundle as the Bott bundle.

    Now let us compare this to the construction of \(\tilde{B}^{n+1} \to \Torus \times \Sphere^n\).
    To this end, let \(\tilde{q} \colon [0,1] \times \Sphere^n \to \Sphere^{n+1}\) be a smooth map implementing the reduced suspension.
    That is, \(\tilde{q}(\{0,1\} \times \Sphere^n \cup [0,1] \times \{e_{n+1}\}) = \{e_{n+2}\}\) and \(\tilde{q}\) descends to a homeomorphism \(\Torus \wedge \Sphere^{n} \xrightarrow{\cong} \Sphere^{n+1}\).
    In particular, the map \(\tilde{q}\) descends to a map \(q \colon \Torus \times \Sphere^n \to \Sphere^{n+1}\) of degree \(1\).
    We further assume that \(\tilde{q}(1/2, x) = (0, x)\) for all \(x \in \Sphere^n\) and
    \begin{align*}
        \tilde{q}([0,1/2] \times \Sphere^n) &= \{(r,u) \in \Sphere^{n+1} \subset \R \times \R^{n+1} \mid r \geq 0\} \eqcolon \Disk_{\geq} \\
        \tilde{q}([1/2,1] \times \Sphere^n) &= \{(r,u) \in \Sphere^{n+1} \subset \R \times \R^{n+1} \mid r \leq 0\} \eqcolon \Disk_{\leq}
    \end{align*}
    Now let \(B^{n+1} \to \Sphere^{n+1} = \Disk_{\geq} \cup_{\{0\} \times \Sphere^n} \Disk_{\leq}\) be the Bott bundle as above.
    Explicitly, the Bott bundle is obtained by clutching two copies of the trivial \(\CCl_{0,n+1}\)-bundle over $\Disk^{n+1}$ along the equator \(\Sphere^n = \{0\} \times \Sphere^{n} \subset \Sphere^{n+1}\) using \(\gamma\).
    Then the pullback bundle \(\tilde{q}^\ast B^{n+1}\) can be understood as gluing the trivial \(\CCl_{0,n+1}\)-bundle on \(( [0,1/2] \sqcup [1/2,1] ) \times \Sphere^n\) at \(\{1/2\} \times \Sphere^n\) using \(\gamma\).
    In these coordinates the bundle is trivial near \(\{0,1\} \times \Sphere^n\) and thus this descends to a description of \(q^\ast B^{n+1} \to \Torus \times \Sphere^n\).
    Passing instead to a single trivialization on \([0,1] \times \Sphere^n\) and identifying \(0 \sim 1\), this yields an identification \(q^\ast B^{n+1} \cong \tilde{B}^{n+1}\). 
    Since \(q\) has degree one, we use \labelcref{eq:bott-degree-formula} to see that \([\tilde{B}^{n+1}] = q^\ast [B^{n+1}] = q^\ast \beta_{\Sphere^{n+1}} = \deg(q) \beta_{\Torus \times \Sphere^n}\).
    In particular, \(\innp{[\Dirac_{\Torus \times \Sphere^n}]}{[\tilde{B}^{n+1}]} = \innp{[\Dirac_{\Torus \times \Sphere^n}]}{\beta_{\Torus \times \Sphere^n}} = 1\).
    Non-vanishing of the family index class is now a consequence of \cref{lem:total_nonvanish_implies_family-nonvanish}.
\end{proof}

\section{A proof of Llarull's theorem in all dimensions} \label{sec:llarull-classical}

For the geometric application, we also need the curvature of the connection family induced by the gauge transform \(\gamma\).
This is the same computation as in \textcite{SpectralFlowLlarull}, which we reproduce here in our setting.

\begin{lemma} \label{lem:gauge_connection}
Let \(\tilde{B}^{n+1}_{\bullet} \to \Sphere^{n}\) from \cref{defi:mapping_torus_bundle} be endowed with the family of connections which descends from the following formula
\begin{align}\label{eq:all-fam_of_connections}
    \tilde{\nabla} \colon [0,1] &\to \mathcal{A}(\Sphere^n \times \CCl_{0,n+1} \to \Sphere^n) \\
    s &\mapsto \tilde{\nabla}_s \coloneq d + s\gamma^{-1}d\gamma. \nonumber
\end{align}
Then the curvature of \(\tilde{\nabla}_s\) is given by
\begin{equation}\label{eq:hc-sphere-curvature}
  \tilde{\FullCurv}_s = -s(1-s)\omega^2 \in \Dforms^2(\Sphere^{n};\Cl_{0,{n+1}}^{(0)})
\end{equation}
where \(\omega \coloneq \gamma^{-1}\D \gamma\in \Dforms^1(\Sphere^{n};\Cl_{0,{n+1}}^{(0)})\) is the Maurer--Cartan form. In particular, if \(v,w \in \T \Sphere^n\) are orthonormal, then, as an endomorphism of \(\CCl_{0,n+1}\), we have that
\[
    \tilde{\FullCurv}_s(v,w)=2s(1-s)\clms(v)\clms(w),
\]
where \(\clms(\blank)\) denotes the left multiplication action of \(\CCl_{0,n+1}\) on itself.
\end{lemma}
\begin{proof}
Identifying \(\Aut^{(0)}(\Cl_{0,n+1}) = (\Cl_{0,n+1}^{(0)})^\times\), we can view \(\gamma\) as a map \(\gamma \colon \Sphere^n \to (\Cl_{0,n+1}^{(0)})^\times\), \(x \mapsto e_{n+1} x \) and \(\gamma(x)^{-1}=x e_{n+1}\), so the Maurer--Cartan form
\[
  \omega\coloneq \gamma^{-1}\D \gamma\in \Dforms^1(\Sphere^{n};\Cl_{0,n+1}^{(0)})
\]
satisfies
\begin{equation}\label{eq:hc-omega-sphere}
  \omega_x(v)=xv,
  \qquad v\in \T_x\Sphere^{n}.
\end{equation}
We compute
\begin{align*}
  \D\omega
  &=\D(\gamma^{-1}\D \gamma) \\
  &=(\D \gamma^{-1})\wedge \D \gamma+\gamma^{-1}\D^2 \gamma \\
  &=-\gamma^{-1}(\D \gamma)\gamma^{-1}\wedge \D \gamma \\
  &=-(\gamma^{-1}\D \gamma)\wedge(\gamma^{-1}\D \gamma) \\
  &=-\omega\wedge\omega,
\end{align*}
therefore
\begin{equation*}%
  \tilde{\FullCurv}_s = s\D\omega+s^2\omega\wedge\omega
  =-s(1-s)\omega^2.
\end{equation*}
If \(v,w\in \T_x\Sphere^{n}\) are orthonormal, then
\[
  \omega_x^2(v,w)=-2vw.
\]
Therefore, as an endomorphism of \(\CCl_{0,n+1}\),
\[
  \tilde{\FullCurv}_s(v,w)=2s(1-s)\clms(v)\clms(w). \qedhere
\]
\end{proof}

\begin{theorem}\label{thm:all-Llarull_sphere}
    Let $(M,g)$ be a closed connected Riemannian spin manifold of dimension $n \geq 2$ with $\scal_{g} \geq n(n-1)$.
    Then any smooth area non-increasing map $f \colon (M,g) \to (\Sphere^n, g_{\mathrm{\circ}})$ of non-zero degree is an isometry for \(n \geq 3\) and an area-density-preserving diffeomorphism for \(n = 2\).
\end{theorem}
\begin{proof}

Consider the \(\Torus\)-family of bundles \(\tilde{B}^{n+1}_{\bullet} \to \Sphere^{n}\) from \cref{defi:mapping_torus_bundle}.
By construction, the family of connections
\begin{align*}
    \tilde{\nabla} \colon [0,1] &\to \mathcal{A}(\Sphere^n \times \CCl_{0,n+1} \to \Sphere^n) \\
    s &\mapsto \tilde{\nabla}_s \coloneq d + s\gamma^{-1}d\gamma \nonumber
\end{align*}
descends to a \(\Torus\)-parametrized family of connections on $\tilde{B}^{n+1}_{\bullet}$.
Let \(E_\bullet = f^\ast \tilde{B}^{n+1}_\bullet\) be the pullback of this family equipped with the pullback family of connections.
Using \labelcref{eq:S^1timesS^n-bott-bundle} from \cref{lem:family_index_main_computation} together with \labelcref{eq:bott-degree-formula}, we obtain
\[
    \begin{aligned}
        \innp{[\Dirac_{\Torus \times M} ]}{(\id \times f)^\ast [\tilde{B}^{n+1}]} & = \innp{[\Dirac_{\Torus \times M} ]}{(\id \times f)^\ast \beta_{\Torus \times \Sphere^n}} \\
        &= \deg(f) \innp{[\Dirac_{\Torus \times M} ]}{\beta_{\Torus \times M}} \\
        &= \deg(f).
    \end{aligned}
\]
Thus \cref{lem:total_nonvanish_implies_family-nonvanish} implies
\( \ind_{\Torus}(\Dirac_{M, E_{\bullet}}) \neq 0 \)
and so by \cref{lem:invertible_everywhere_vanishes} there exists an \(s \in [0,1]\) such that \(\Dirac_{s}\) has a non-trivial kernel, where \(\Dirac_s\) is the Dirac operator acting on sections of \(\SpinBdl_{M} \otimes \CCl_{0,n+1}\) defined with respect to the connection \(\nabla_{M,s} \coloneq \nabla_{M} \otimes f^\ast \tilde{\nabla}_{s}\) and \(\nabla_{M}\) is the spinor connection on \(\SpinBdl_M\).

From here onwards the proof proceeds as in \cite{SpectralFlowLlarull,baer2025spectralflowatiyahpatodisingerindex}.
Indeed, we have the Schrödinger--Lichnerowicz formula for the twisted Dirac operator \(\Dirac_s\),
\begin{align}\label{eq:all-SL-formula}
    (\Dirac_s)^2 = \nabla_{M,s}^* \nabla_{M,s} + \frac{\scal_g}{4} + \Curv_s,
\end{align}
where \(\Curv_s = \clm\big(f^*\tilde{\FullCurv}_s\big)\).
Choosing orthonormal bases \(w_1, \dotsc, w_n\) of \(\T_{p} M\) and \(v_1, \dotsc, v_n\) of \(\T_{f(p)} \Sphere^n \subset \R^{n+1}\) such that \(\D f(w_i) = \lambda_i v_i\) with \(\lambda_i \geq 0\), we can write
\begin{align*}
\Curv_s &= \sum_{i < j} \clm(w_i) \clm(w_j) \otimes f^\ast \tilde{\FullCurv}_s(w_i, w_j) \\
 &= 2 s (1-s) \sum_{i < j} \clm(w_i) \clm(w_j) \otimes \clms(\D f(w_i)) \clms(\D f(w_j)) \\
 &= 2 s (1-s) \sum_{i < j} \lambda_i \lambda_j \clm(w_i) \clm(w_j) \otimes \clms(v_i) \clms(v_j),
\end{align*} 
where we have used \cref{lem:gauge_connection}.
Thus, using that \(\lambda_i \lambda_j \leq 1\) for all \(i \neq j\) due to the area non-increasing assumption,
\begin{equation}
    \innp{\Curv_s u}{u}_p \geq - 2 s(1-s) \sum_{i<j} \lambda_i \lambda_j \abs{u}_p^2 \geq -s(1-s) n(n-1) \abs{u}_p^2 \geq - \frac{n(n-1)}{4} \abs{u}_p^2, \label{eq:curvature_estimate_llarull}
\end{equation}
where if equality holds and \(\abs{u}_p \neq 0\), then \(s = \frac{1}{2}\) and \(\lambda_i \lambda_j = 1\) for all \(i < j\).
If \(n \geq 3\), the latter, in turn, is equivalent to \(\lambda_i = 1\) for all \(i\), that is, \(\D_p f\) is an isometry.

Now take a non-zero section \(u\) of \(\SpinBdl_{M} \otimes \CCl_{0,n+1}\) such that \(\Dirac_s u = 0\).
Then using the hypothesis \(\scal_g \geq n(n-1)\) together with \labelcref{eq:curvature_estimate_llarull} yields
\begin{align*}
0 = \int_{M} \abs{\Dirac_s u}^2 \dV &= \int_{M} \abs{\nabla_{M,s} u}^2 \dV + \int_{M} \frac{\scal_g}{4} \abs{u}^2 + \innp{\Curv_s u}{u} \dV \\
&\geq  \int_{M} \abs{\nabla_{M,s} u}^2 \dV \geq 0.
\end{align*}
Thus \(\nabla_{M,s} u = 0\), so \(u\) vanishes nowhere, and we are in the equality case of \labelcref{eq:curvature_estimate_llarull} at every point \(p \in M\).
By the discussion in the previous paragraph this means that \(s = \frac{1}{2}\) and \(f\) is a local diffeomorphism.
Since \(\Sphere^n\) is simply-connected, \(f\) must be a diffeomorphism.
Moreover, the discussion also shows that for \(n \geq 3\), \(f\) is a local isometry, and for \(n = 2\), it at least preserves the area density.
\end{proof}

\section{Products of convex hypersurfaces} \label{sec:products}

The proof of Llarull's theorem from the previous section generalizes to the case where the comparison manifold is a product of hypersurfaces (in particular, to a product of spheres).
We start with a lemma that contains the linear algebraic core of the rigidity conclusion.

\begin{lemma}
\label{lem:linear-algebra-factorwise-area-rigidity}
    Let \(W=\bigoplus_{j=1}^k W_j\) be an orthogonal direct sum of finite-dimensional Euclidean vector spaces with \(n_j \coloneq \dim W_j\geq2\), \(k \geq 1\), and set \(n\coloneqq \dim W\).
    Let \(V\) and \(Z\) be finite-dimensional Euclidean vector spaces and let \(F = (L,K) \colon V \to W \oplus Z\) be a linear map.
    For each \(j \in \{1, \dotsc k\}\) choose an orthonormal basis \(\epsilon_{j,1},\dots,\epsilon_{j,n_j}\) of \(W_j\).
    
    Assume \( \abs{\Lambda^2 F}_{\op} \leq 1 \) and that for every \(j\) and every \(1 \leq \alpha < \beta \leq n_j\), we have \(\abs{\Lambda^2 L^\top(\epsilon_{j,\alpha}\wedge \epsilon_{j,\beta})}=1\), where \(L^\top \colon W \to V\) denotes the adjoint of \(L\).
    Then \(L\) is surjective and \(K|_{(\ker L)^\perp} = 0\).
    Moreover, if \(n \geq 3\), then \(L|_{(\ker L)^\perp}\colon (\ker L)^\perp\to W\) is an isometry, and if \(n = 2\), then \(\Lambda^2 L|_{(\ker L)^\perp}\colon \Lambda^2 (\ker L)^\perp\to \Lambda^2 W\) is an isometry.
\end{lemma}

\begin{proof}
    Define \(P \coloneq L \circ L^\top \colon W \to W\).
    Then \( \abs{\Lambda^2 F}_{\op} \leq 1 \) implies \( \abs{\Lambda^2 L}_{\op} \leq 1 \) and so \(0 \leq \Lambda^2 P \leq \id_{\Lambda^2 W}\).
    Using this, we have the following inequality
    \[
        \begin{aligned}
            1 &= \abs{\Lambda^2 L^\top (\epsilon_{j,\alpha} \wedge \epsilon_{j,\beta})}^2 \\
            &= \innp{\Lambda^2 P(\epsilon_{j,\alpha} \wedge \epsilon_{j,\beta})}{\epsilon_{j,\alpha} \wedge \epsilon_{j,\beta}}
            \leq \abs{\epsilon_{j,\alpha} \wedge \epsilon_{j,\beta}}^2 = 1
        \end{aligned}
    \]
    which turns out to be an equality.
    Therefore for every \(j\) and every \(1 \leq \alpha < \beta \leq n_j\),
    \begin{equation}
        \label{eq:lambda-two-p-fixes-coordinate-bivectors}
        \Lambda^2 P(\epsilon_{j,\alpha} \wedge \epsilon_{j,\beta}) = \epsilon_{j,\alpha} \wedge \epsilon_{j,\beta}.
    \end{equation}
    In particular, every two-plane of the form \(\linspan \{\epsilon_{j,\alpha}, \epsilon_{j,\beta}\} \subset W_j\) for \(\alpha < \beta\) is preserved by \(P\).
    If \(n_j = 2\), this implies that \(W_j\) is \(P\)-invariant and \(\det(P|_{W_j}) = 1\).
    For \(n_j \geq 3\), we also conclude that \(W_j\) is \(P\)-invariant and for each collection of three pairwise distinct indices \(1 \leq \alpha,\beta,\gamma \leq n_j\) we have 
    \[
        P \epsilon_{j,\alpha} \in \linspan \{\epsilon_{j,\alpha}, \epsilon_{j,\beta}\} \cap \linspan \{\epsilon_{j,\alpha}, \epsilon_{j,\gamma}\} = \linspan \{\epsilon_{j,\alpha}\}.
    \]
    Thus if \(n_j \geq 3\), every \(\epsilon_{j,\alpha}\) for \(1 \leq \alpha \leq n_j\) is an eigenvector of \(P\).
    Together with \labelcref{eq:lambda-two-p-fixes-coordinate-bivectors} this means that \(P|_{W_j} = \id_{W_j}\) whenever \(n_j \geq 3\).
    In particular, we obtain in total \(\det(P) = \prod_{j=1}^k \det(P|_{W_j}) = 1\) and \(P = L \circ L^\top\) is invertible. 
    This implies that \(L\) is surjective.
    
    Next we prove that \(L|_{(\ker L)^\perp}\colon (\ker L)^\perp\to W\) is an isometry if \(n \geq 3\).
    To this end, let \(\sigma_1 \geq \dotsc \geq \sigma_n \geq 0\) be the eigenvalues of \(P\) and note that by \(\Lambda^2 P \leq \id_{\Lambda^2 W}\), we have \(\sigma_1 \sigma_2 \leq 1\).
    We claim that \(\sigma_1 \leq 1\).
    Indeed, if \(\sigma_1 > 1\), then \(\sigma_2 \leq \sigma_1^{-1} < 1\) and provided \(n \geq 3\), we get that
    \[
    1 = \det(P) = \sigma_1 \sigma_2 \dotsm \sigma_n \leq \sigma_1 \sigma_2^{n-1} \leq \sigma_1^{2-n} < 1,
    \]
    which is a contradiction.
    Hence, \(\sigma_1 \leq 1\) and together with \(\det(P) = 1\), we get that \(\sigma_1 = \sigma_2 = \dotsc = \sigma_n = 1\).
    Thus \(L \circ L^\top = P = \id_W\) which implies that \(L^\top\colon W\to V\) is an isometric embedding with \(\im(L^\top) = (\ker L)^\perp\), and the restriction \( L\colon(\ker L)^\perp\to W \) is an isometry.
    
    In the case \(n = 2\), we still have \(\Lambda^2 L \circ (\Lambda^2 L)^\top = \Lambda^2 (L \circ L^{\top}) = \Lambda^2 P = \id_{\Lambda^2 W}\), so \(\Lambda^2 L^{\top} \colon \Lambda^2 W \to \Lambda^2 V\) is an isometric embedding and \(\Lambda^2 L|_{(\ker L)^\perp} \colon \Lambda^2 (\ker L)^\perp \to \Lambda^2 W\) is an isometry.
    
    It remains to show that \( K \) vanishes on \( (\ker L)^\perp \).
    To this end, define \( B\coloneqq K\circ L^\top\colon W\to Z \).
    The map \( F \circ L^\top \colon W \to W \oplus Z\) is given by \( w\mapsto(P w, Bw) \).
    Using that \( \Lambda^2 L^\top \) is an isometric embedding and \( \abs{\Lambda^2 F}_{\op} \leq 1\), we see that the map \(F \circ L^\top\) is area non-increasing.
    Then let \( u,v\in W \) be an orthonormal pair of vectors and keeping \(\Lambda^2 P = \id_{\Lambda^2 W}\) in mind, we derive the equality
    \[ 
        \abs{(P u,Bu)\wedge(P v,Bv)}^2 = \underbrace{\abs{Pu\wedge Pv}^2}_{=\abs{u \wedge v} = 1} + \abs{Pu \wedge Bv - Pv \wedge Bu}^2 + \abs{Bu\wedge Bv}^2.
    \]
    The area non-increasing property of \(F \circ L^\top\) gives \( \abs{(P u,Bu)\wedge(P v,Bv)}^2\leq 1 \), so we must have \(Pu \wedge Bv = Pv \wedge Bu \in \Lambda^2 (W \oplus Z)\).
    Since \(P u \wedge P v = u \wedge v \neq 0\), the vectors \(Pu\) and \(Pv\) are linearly independent and thus \( Bu = Bv = 0\).
    This proves that \(K \circ L^\top = B=0\) because \(\dim (W) \geq 2\) and so \( K|_{(\ker L)^\perp}=0\).
\end{proof}

We now start with the proof of \cref{thm:general-product-splitting} stated in the introduction.

\begin{proof}[Proof of \cref{thm:general-product-splitting}]
    For the index-theoretic part of the proof, the circle factors will be treated in the same way as the other factors.
    To this end, define \(N_{k+i} \coloneq \Sphere^1_{r_i} \subset \R^2\) and \(n_{k+i} \coloneq 1\) for \(i \in \{1, \dotsc, l\}\).
    For each $j \in \{1, \dotsc, k+l\}$, let $\nu_j \colon N_j \to \Sphere^{n_j}$ be the Gauß map of $N_j$ corresponding to the outward pointing unit normal.
    Note that each \(\nu_j\) is a diffeomorphism by the strict convexity assumption.
    Consider the vector bundles $\tilde{E}_{j} \coloneq (\id \times \nu_j)^*\tilde{B}^{n_j+1} \to \Torus \times N_j$, where the bundle \(\tilde{B}^{n_j+1} \to \Torus \times \Sphere^{n_j}\) is defined in \cref{defi:mapping_torus_bundle}.
    By \cref{lem:family_index_main_computation} and \labelcref{eq:bott-degree-formula}, we have that $[\widetilde{E}_j] = \beta_{\Torus \times N_j} \in \K^{n_j+1}(\Torus \times N_j)$ for all $j\in \{1, \dotsc, k+l\}$.
    Let $\tilde{E} \coloneq (\bigboxtimes_{j=1}^{k+l}\widetilde{E}_j) \to \Torus^{k+l} \times N \times T$ denote the exterior tensor product of these bundles.
    We work with the graded tensor product of Clifford modules and thereby \(\tilde{E}\) becomes a bundle of graded \(\CCl_{0,(n+l)+(k+l)}\)-modules.
    Then by \labelcref{eq:bott-product}, we have \([\tilde{E}] = \beta_{\Torus^{k+l} \times N \times T} \in \K^{(n+l)+(k+l)}(\Torus^{k+l} \times N \times T)\).
    Now let \(E \coloneq (\id_{\Torus^{k+l}} \times f)^\ast \tilde{E} \to \Torus^{k+l} \times M\) and using \labelcref{eq:bott-degree-formula}, we see that
        \begin{align*}
        [E] = \deg(f)\ \beta_{\Torus^{k+l} \times M} \in \K^{(n+l)+(k+l)}(\Torus^{k+l} \times M).
    \end{align*}
    Since $f$ has non-zero degree, we finally arrive at
    \begin{equation}
        \innp{[\Dirac_{\Torus^{k+l} \times M}]}{[E]} = \deg(f) \innp{[\Dirac_{\Torus^{k+l} \times M}]}{\beta_{\Torus^{k+l} \times M}} = \deg(f) \neq 0.
    \end{equation}
    
    We now view \(E\) as a \(\Torus^{k+l}\)-family of \(\CCl_{0,(n+l)+(k+l)}\)-module bundles over \(M\) denoted by \(E_{\bullet} \to M\).
    Endow \(E_{\bullet}\) with the family of connections obtained by the pullback along \((\nu_1 \times \dotsm \times \nu_{k+l}) \circ f\) of the product of the connections \labelcref{eq:all-fam_of_connections}.
    Then the family index $\ind_{\Torus^{k+l}}(\Dirac_{M,E_{\bullet}}) \in {\K^{k+l}(\Torus^{k+l})}$ is non-zero by \cref{lem:total_nonvanish_implies_family-nonvanish}. 
    
    Set \(p \coloneq \proj_N \circ f\) and \(\mathcal{H} \coloneq \ker(\D p)^\perp\).
    Given \(x \in M\) and a vector \(\xi \in \T_{p(x)} N\), we write
    \[
        \xi^{\mathcal{H}} \coloneq (\D p)^\top(\xi) \in \mathcal{H}_{x},
    \]
    where \((\D p)^{\top} \colon p^\ast \T N \to \T M\) denotes the adjoint of \(\D p\) with respect to the Riemannian metrics.
    This pointwise definition makes sense even if we do not know (yet) that \(p\) is a submersion.
    
    Let \(\Shape_j = \D \nu_j \in \End(\T N_j)\) be the shape operator of \(N_j\), \(j \in \{1, \dotsc, k+l\}\).
    By the strict convexity assumption, we can choose at each point \(y_j \in N_j\) an orthonormal basis \(\epsilon_{j,1}, \dotsc, \epsilon_{j,n_j}\) of \(\T N_j\) such that \(\Shape_j \epsilon_{j,\alpha} = \kappa_{j, \alpha} \epsilon_{j,\alpha}\) with \(\kappa_{j,\alpha} > 0\).
    We note that at the point \((y_1, \dotsc, y_{k+l}) \in N \times T\), the Gauss equation implies that
    \begin{equation}
        \scal_{g_N \oplus g_T} = \scal_{g_N} \circ \proj_N = 2 \sum_{j = 1}^{k} \sum_{1 \leq \alpha < \beta \leq n_j} \kappa_{j, \alpha} \kappa_{j, \beta}. \label{eq:scalar_via_principal-curvatures}
    \end{equation}
    Then the curvature of the pullback of the connection family \labelcref{eq:all-fam_of_connections} along \(\nu_j\) satisfies
    \begin{equation}\label{eq:pullback_curvature_via_principal-curvatures}
        \begin{aligned}
            \FullCurv^{\nu_j^\ast \tilde{\nabla}_{s_j}}(\epsilon_{j,\alpha}, \epsilon_{j,\beta}) &= \tilde{\FullCurv}_{s_j}(\Shape_j(\epsilon_{j,\alpha}), \Shape_j(\epsilon_{j,\beta})) \\
            &= \kappa_{j, \alpha} \kappa_{j,\beta} \tilde{\FullCurv}_{s_j}(\epsilon_{j, \alpha}, \epsilon_{j,\beta}) \underset{\eqref{eq:hc-sphere-curvature}}{=} 2 s_j (1-s_j) \kappa_{j,\alpha} \kappa_{j,\beta} \clms(\epsilon_{j,\alpha}) \clms(\epsilon_{j,\beta})
        \end{aligned}
    \end{equation}
    for \(j \leq k\), \(\alpha < \beta\) and \(s_j \in [0,1]\), where \(\clms\) denotes left multiplication on \(\CCl_{0,n_j+1}\).
    For \(j > k\) we note that \(\FullCurv^{\nu_j^\ast \tilde{\nabla}_{\bullet}} = 0\) because these correspond to the one-dimensional factors.
    
    By \cref{lem:invertible_everywhere_vanishes} there exists a point \(\mathbf{s} = [(s_1, \dotsc, s_{k+l})] \in \Torus^{k+l} = [0,1]^{k+l} / \sim\) such that \(\Dirac_{\mathbf{s}} \coloneq \Dirac_{M, E_{\mathbf{s}}}\) has a non-trivial kernel.
    Choose a smooth section \(u\) of \(\SpinBdl_{M} \otimes \CCl_{0,(n+l)+(k+l)}\) such that \(\Dirac_{\mathbf{s}} u = 0\).
    Let \(\nabla_{\mathbf{s}}\) denote the corresponding connection on \(\SpinBdl_{M} \otimes \CCl_{0,(n+l)+(k+l)}\), that is,
    \[
        \nabla_{\mathbf{s}} \coloneq \nabla_M \otimes \nabla_{E_{\mathbf{s}}} \coloneq \nabla_M \otimes f^\ast \bigboxtimes_{j=1}^{k+l}\nu_{j}^\ast \tilde{\nabla}_{s_j}.
    \]
    The Schrödinger-Lichnerowicz formula for \(\Dirac_\mathbf{s}\) is of the form
    \begin{align}
        \Dirac_{\mathbf{s}}^2 = {\nabla}_{\mathbf{s}}^*{\nabla}_{\mathbf{s}} + \frac{\scal_g}{4} + \Curv_{\mathbf{s}},
    \end{align}
    where \(\Curv_{\mathbf{s}} = \clm(\FullCurv_{E_\mathbf{s}}) = \sum_{j=1}^k \Curv^{(j)}_{s_j}\) with
    \begin{equation} \label{eq:curvature_summand}
        \begin{aligned}
            \Curv^{(j)}_{t} \coloneq& \sum_{1 \leq \alpha < \beta \leq n_j}\clm(\epsilon_{j,\alpha}^{\mathcal{H}} \wedge \epsilon_{j,\beta}^{\mathcal{H}}) \FullCurv^{\nu_j^\ast\tilde{\nabla}_{t}}(\epsilon_{j,\alpha}, \epsilon_{j,\beta}) \\
            \underset{\eqref{eq:pullback_curvature_via_principal-curvatures}}{=}& 2 t (1-t) \sum_{1 \leq \alpha < \beta \leq n_j} \kappa_{j,\alpha} \kappa_{j,\beta}\ \clm(\epsilon_{j,\alpha}^{\mathcal{H}} \wedge \epsilon_{j,\beta}^{\mathcal{H}}) \clms_j(\epsilon_{j,\alpha}) \clms_j(\epsilon_{j,\beta})
        \end{aligned}
    \end{equation}
    Here \(\clms_{j}\) denotes the Clifford action by left multiplication on the \(j\)-th Clifford algebra tensor factor (and the identity on all other factors) with respect to the graded tensor product decomposition 
    \[
        \SpinBdl_{M} \otimes \CCl_{0,(n+l)+(k+l)} = \SpinBdl_{M} \otimes \left(\bigotimes_{j=1}^k \CCl_{0,n_j+1} \right) \otimes \CCl_{0,l+l}.
    \]
    Moreover, \(\clm(\blank)\) in \labelcref{eq:curvature_summand} is to be understood as the Clifford action of tangent vectors and their exterior products on the \(\SpinBdl_M\) tensor factor.
    Thus, with \(\zeta_{j,\alpha,\beta} \coloneq \epsilon_{j,\alpha}^{\mathcal{H}} \wedge \epsilon_{j,\beta}^{\mathcal{H}}\) and \(A_{j,\alpha,\beta} \coloneq \clm(\zeta_{j,\alpha,\beta})\clms_j(\epsilon_{j,\alpha})\clms_j(\epsilon_{j,\beta})\), we obtain that
    \begin{align}
0 &= \int_{M} \abs{\Dirac_{\mathbf{s}} u}^2 \dV \nonumber\\ &
= \int_{M} \abs{\nabla_{\mathbf{s}} u}^2 \dV + \int_{M} \frac{\scal_g}{4} \abs{u}^2 + \innp{\Curv_{\mathbf{s}} u}{u} \dV \nonumber\\
&=  \int_{M} \abs{\nabla_{\mathbf{s}} u}^2 \dV + \frac{1}{2} \int_{M} \abs{u}^2 \sum_{\substack{j \in \{1, \dotsc, k\} \\ 1 \leq \alpha < \beta \leq n_j}} \kappa_{j,\alpha} \kappa_{j,\beta} \left(1 - 4 s_j (1-s_j) \abs{\zeta_{j,\alpha,\beta}}\right) \dV \label{eq:main_rigidity_estimate} \\
&\quad + \frac{1}{4} \int_{M} \abs{u}^2  (\scal_g - \scal_{g_N} \circ \proj_{N} \circ f ) \dV \label{eq:scalar_curvature_defect} \\
&\quad + 2 \int_{M} \sum_{\substack{j \in \{1, \dotsc, k\} \\ 1 \leq \alpha < \beta \leq n_j}} s_j (1-s_j) \kappa_{j,\alpha} \kappa_{j,\beta} \innp{\abs{\zeta_{j,\alpha,\beta}}u + A_{j,\alpha,\beta}u}{u} \dV. \label{eq:clifford_action_defect}
\end{align}
    Since \(f\) is area non-increasing, the component \(p\) is area non-increasing and hence \(\abs{\Lambda^2(\D p)^\top}_{\op}\leq 1\).
    Thus
    \[
        \abs{\zeta_{j,\alpha,\beta}}
        = \abs{\Lambda^2(\D p)^\top(\epsilon_{j,\alpha} \wedge \epsilon_{j,\beta})}
        \leq \abs{\epsilon_{j,\alpha} \wedge \epsilon_{j,\beta}} = 1.
    \] 
    Together with \(0 \leq 4 t (1-t) \leq 1\) for all \(t \in [0,1]\), we see that each summand under the second integral in \labelcref{eq:main_rigidity_estimate} is non-negative.
    The integrand in \labelcref{eq:scalar_curvature_defect} is non-negative by the scalar curvature lower bound in the hypothesis of the theorem.
    The summands of the integrand in \labelcref{eq:clifford_action_defect} are non-negative by the Cauchy--Schwarz inequality because \(\abs{A_{j,\alpha,\beta}} = \abs{\zeta_{j,\alpha,\beta}}\), using that for a simple bivector \(\zeta = \xi_1 \wedge \xi_2\) we have \(\clm(\xi_1 \wedge \xi_2)^2 = -(\abs{\xi_1}^2 \abs{\xi_2}^2 - \innp{\xi_1}{\xi_2}^2) \id = -\abs{\xi_1 \wedge \xi_2}^2 \id \).

    We thus conclude that all non-negative summands in \labelcref{eq:scalar_curvature_defect,eq:main_rigidity_estimate,eq:clifford_action_defect} must vanish identically, in particular \(\nabla_{\mathbf{s}} u =0\).
    Thus \(\abs{u} \neq 0\) everywhere and since the \(\kappa_{j,\alpha}\) are strictly positive, we further conclude from vanishing of \labelcref{eq:main_rigidity_estimate} that \(s_j = \frac{1}{2}\) and 
    \begin{equation}
        \label{eq:factorwise-area-equality}
        \abs{\epsilon_{j,\alpha}^{\mathcal{H}} \wedge \epsilon_{j,\beta}^{\mathcal{H}}} = \abs{\zeta_{j,\alpha,\beta}} = 1
    \end{equation}
    for all \(j \in \{1, \dotsc, k\}\), \(1 \leq \alpha < \beta \leq n_j\).
    Then vanishing of \labelcref{eq:scalar_curvature_defect,eq:clifford_action_defect} also implies
    \begin{gather}
        \scal_g = \scal_{g_N} \circ \proj_N \circ f, \label{eq:scalar_curv_equal}\\
        A_{j,\alpha,\beta}u = -u \quad \forall j \in \{1, \dotsc, k\},\ 1 \leq \alpha < \beta \leq n_j, \label{eq:clifford_two-form_identity}
    \end{gather}
    respectively.
    
    Having established these identities, we can draw the first rigidity conclusion which we note in the proposition below.

    \begin{proposition} \label{prop:submersion_package}
        The map \(p \colon (M,g) \to (N,g_N)\) is a submersion and \(\D(\proj_T \circ f)|_{\mathcal{H}} = 0\).
        Moreover, if \(n \geq 3\), then \(p\) is a Riemannian submersion, and if \(n = 2\), then \(\Lambda^2 \D p|_{\mathcal{H}} \colon \Lambda^2 \mathcal{H} \to \Lambda^2 \T N\) is pointwise an isometry.
    \end{proposition}
    \begin{proof}
        In the case \(n = 0\), the proposition is vacuous.
        Otherwise, since \(\zeta_{j,\alpha,\beta} = \Lambda^2(\D p)^\top(\epsilon_{j,\alpha} \wedge \epsilon_{j,\beta})\), identity \labelcref{eq:factorwise-area-equality}, together with the fact that \(f\) is area non-increasing, ensures that the map \((\D p,\D(\proj_T \circ f)) : \T M \to \T N \oplus \T T\) pointwise satisfies the hypotheses of \cref{lem:linear-algebra-factorwise-area-rigidity}, which implies the statement. 
    \end{proof}
    
    Our next goal is to show that \(p\) actually splits the manifold as a product.
    This is a product rigidity statement in the spirit of \textcite{riedler-tony} and we also follow their strategy of first establishing a Ricci curvature identity (compare~\cite[Proposition~5.4]{riedler-tony}).
    \begin{proposition} \label{prop:ricci-curvature-formula}
        If \(n \geq 3\), then
        \begin{equation}
            \Ric_{(M,g)} = \Ric_{(N,g_N)} \circ \hspace{0.5mm}(\D p \otimes \D p). \label{eq:ricci-identity}
        \end{equation}
        
        If \(n = 2\), then
        \begin{equation}
            \Ric_{(M,g)} =  (\Gauss_{(N,g_N)} \circ \hspace{0.8mm} p) \cdot (g \circ P_{\mathcal{H}} \otimes P_{\mathcal{H}}), \label{eq:gauss-identity}
        \end{equation}
        where \(P_{\mathcal{H}} \colon \T M \to \mathcal{H}\) denotes the orthogonal projection onto the horizontal bundle.
        
        If \(n = 0\), then \(\Ric_{(M,g)} = 0\).
    \end{proposition}
    \begin{proof}
        We will write \(\Ric^{\sharp}\) for the Ricci curvature as a \((1,1)\)-tensor.
        Then \labelcref{eq:ricci-identity,eq:gauss-identity} are equivalent to
        \begin{equation}
            \Ric_M^\sharp(X) = \left( \Ric^\sharp_N(\D p(X)) \right)^{\mathcal{H}} \label{eq:ricci-sharped-identity}
        \end{equation}
        and
        \begin{equation}
            \Ric_M^\sharp(X) = \Gauss_{N}(p(x))\cdot P_{\mathcal{H}}(X), \label{eq:gauss-sharped-identity}
        \end{equation}
        respectively, for all \(x \in M\), \(X \in \T_{x} M\).
        These are the conditions we will verify.
        
        Indeed, fix \(X \in \T_x M\).
        Then for each \(j \in \{1, \dotsc, k\}\) consider again the principal curvature directions \(\epsilon_{j,\alpha} \in \T_{\proj_{N_j}(p(x))} N_j \subset \T_{p(x)} N\) of \(N_j\), \(1 \leq \alpha \leq n_j\); that is, the eigenbasis of the shape operator \(\Shape_j = \D \nu_j\) chosen above.
        In the case \(n \geq 3\) we define \(e_{j,\alpha} \coloneq \epsilon_{j,\alpha}^{\mathcal{H}} \in \mathcal{H}_x\) for \(j \in \{1, \dotsc, k\}\) and \(1 \leq \alpha \leq n_j\). By \cref{prop:submersion_package}, \(p\) is a Riemannian submersion if $n \geq 3$, and hence \(\{e_{j,\alpha}\}\) forms an orthonormal basis of \(\mathcal{H}_x\). In the case \(n = 2\), we necessarily have \(k=1\) and \(n_1=2\). We then fix an orthonormal basis \(e_{1,1}, e_{1,2} \in \mathcal{H}_x\) of \(\mathcal{H}_x\) such that \(e_{1,1} \wedge e_{1,2} = \zeta_{1,1,2} = \epsilon_{1,1}^{\mathcal{H}} \wedge \epsilon_{1,2}^{\mathcal{H}}\), which is possible by \labelcref{eq:factorwise-area-equality}.
        Using the Ricci curvature formula in terms of the spinorial connection~\cite[Corollary~2.8]{spinorial-approach-riemannian}, the fact that \(u\) is \(\nabla_{\mathbf{s}}\)-parallel, and that the twisting curvature \(\FullCurv_{E_{\mathbf{s}}}\) vanishes on vertical vectors by construction, we obtain for \(n \geq 2\), the formula 
        \begin{equation*}
            \begin{aligned}
                \clm&(\Ric_M^\sharp(X)) u =\\
             &= 2 \sum_{\substack{j \in \{1, \dotsc, k\} \\ 1 \leq \alpha \leq n_j}} \clm(e_{j,\alpha}) \FullCurv_{E_{\mathbf{s}}}(X, e_{j,\alpha}) u \\
            &= 2 \sum_{\substack{j \in \{1, \dotsc, k\} \\ 1 \leq \alpha, \beta \leq n_j}} g(X, e_{j,\beta}) \clm(e_{j,\alpha})  \FullCurv_{E_{\mathbf{s}}}(e_{j,\beta}, e_{j,\alpha}) u \\
            &= 2 \sum_{\substack{j \in \{1, \dotsc, k\} \\ 1 \leq \alpha < \beta \leq n_j}} g(X, e_{j,\beta}) \clm(e_{j,\alpha}) \FullCurv_{E_{\mathbf{s}}}(e_{j,\beta}, e_{j,\alpha}) + g(X, e_{j,\alpha}) \clm(e_{j,\beta}) \FullCurv_{E_{\mathbf{s}}}(e_{j,\alpha}, e_{j,\beta}) u \\
            &= 2 \sum_{\substack{j \in \{1, \dotsc, k\} \\ 1 \leq \alpha < \beta \leq n_j}} ( - g(X, e_{j,\beta}) \clm(e_{j,\alpha}) + g(X, e_{j,\alpha}) \clm(e_{j,\beta}) )\FullCurv_{E_{\mathbf{s}}}(e_{j,\alpha}, e_{j,\beta}) u \\
            &= \sum_{\substack{j \in \{1, \dotsc, k\} \\ 1 \leq \alpha < \beta \leq n_j}} \kappa_{j,\alpha} \kappa_{j,\beta}  ( g(X, e_{j,\alpha}) \clm(e_{j,\beta}) - g(X, e_{j,\beta}) \clm(e_{j,\alpha})) \clms_j(\epsilon_{j,\alpha}) \clms_j(\epsilon_{j,\beta}) u,
            \end{aligned}
        \end{equation*}
        where in the last step we have expressed the twisting curvature using \labelcref{eq:pullback_curvature_via_principal-curvatures} and \(s_j = 1/2\) for \(1 \leq j \leq k\).
        
        In the edge case \(n = 0\), the spinorial curvature formula~\cite[Corollary~2.8]{spinorial-approach-riemannian} gives \(\clm(\Ric_{M}^\sharp(X)) u = 0\) directly and thus \(\Ric^\sharp_M(X) = 0\) by invertibility of the Clifford action, thereby already concluding the case \(n = 0\).
        
        Continuing with \(n \geq 2\), we note that \labelcref{eq:clifford_two-form_identity} implies
        \[
            \clms_j(\epsilon_{j,\alpha}) \clms_{j}(\epsilon_{j,\beta}) u = \clm(\zeta_{j,\alpha,\beta}) u = \clm(e_{j,\alpha}) \clm(e_{j,\beta}) u
        \]
        for all \(j \in \{1, \dotsc, k\}\) and all \(1 \leq \alpha < \beta \leq n_j\).
        Thus, we compute
        \begin{align}
            \clm&(\Ric_M^\sharp(X)) u = \nonumber\\
            &= \sum_{\substack{j \in \{1, \dotsc, k\} \\ 1 \leq \alpha < \beta \leq n_j}} \kappa_{j,\alpha} \kappa_{j,\beta}  ( g(X, e_{j,\alpha}) \clm(e_{j,\beta}) - g(X, e_{j,\beta}) \clm(e_{j,\alpha})) \clm(e_{j,\alpha}) \clm(e_{j,\beta}) u \notag \\
            &= \sum_{\substack{j \in \{1, \dotsc, k\} \\ 1 \leq \alpha < \beta \leq n_j}} \kappa_{j,\alpha} \kappa_{j,\beta}
                ( g(X, e_{j,\alpha}) \clm(e_{j,\alpha}) + g(X, e_{j,\beta}) \clm(e_{j,\beta})) u \notag \\
            &= \clm \Big( \sum_{\substack{j \in \{1, \dotsc, k\} \\ 1 \leq \alpha < \beta \leq n_j}}
            \kappa_{j,\alpha} \kappa_{j,\beta} \left( g(X, e_{j,\alpha}) e_{j,\alpha} + g(X, e_{j,\beta}) e_{j,\beta} \right) \Big) u. \label{eq:case-distinction-breakpoint}\\
            \intertext{Now we first treat the case \(n \geq 3\), where \(p\) is a Riemannian submersion and \(e_{j,\alpha} = \epsilon_{j,\alpha}^{\mathcal{H}}\), so we further obtain that}
            &= \clm \Big( \sum_{\substack{j \in \{1, \dotsc, k\} \\ 1 \leq \alpha < \beta \leq n_j}}
            \kappa_{j,\alpha} \kappa_{j,\beta} \left( g_{N}(\D p(X), \epsilon_{j,\alpha}) e_{j,\alpha} + g_{N}(\D p(X), \epsilon_{j,\beta}) e_{j,\beta} \right) \Big) u \notag \\
            &= \clm \Big(  \big( \sum_{\substack{j \in \{1, \dotsc, k\} \\ 1 \leq \alpha < \beta \leq n_j}}
            \kappa_{j,\alpha} \kappa_{j,\beta} \left( g_{N}(\D p(X), \epsilon_{j,\alpha}) \epsilon_{j,\alpha} + g_{N}(\D p(X), \epsilon_{j,\beta}) \epsilon_{j,\beta} \right) \big)^{\mathcal{H}}\Big) u. \notag
        \end{align}
        Then note that, by the Gauss equation, the Ricci curvature of \((N,g_N)\) can be expressed as
        \begin{align*}
            \Ric_{N}^\sharp(Y) &=  \sum_{\substack{j \in \{1, \dotsc, k\} \\ 1 \leq \alpha \neq \beta \leq n_j}}
                \kappa_{j,\alpha} \kappa_{j,\beta}\ g_{N}(Y, \epsilon_{j,\alpha}) \epsilon_{j,\alpha}  \\
            &= \sum_{\substack{j \in \{1, \dotsc, k\} \\ 1 \leq \alpha < \beta \leq n_j}}
                \kappa_{j,\alpha} \kappa_{j,\beta} \left( g_{N}(Y, \epsilon_{j,\alpha}) \epsilon_{j,\alpha} + g_{N}(Y, \epsilon_{j,\beta}) \epsilon_{j,\beta} \right).
        \end{align*}
        Thus 
        \[
             \clm(\Ric_M^\sharp(X))u = \clm \Big( \left( \Ric^\sharp_N(\D p(X)) \right)^{\mathcal{H}} \Big) u,
        \]
        which proves \labelcref{eq:ricci-sharped-identity} by invertibility of the Clifford action and pointwise non-vanishing of \(u\).
        This concludes the case \(n \geq 3\).
        
        To finish the case \(n = 2\), we continue from \labelcref{eq:case-distinction-breakpoint}, which then directly simplifies to
        \begin{equation*}
            \begin{aligned}
                \clm(\Ric_M^\sharp(X)) u &= \clm \Big(\kappa_{1,1} \kappa_{1,2}\ (g(X,e_{1,1}) e_{1,1} + g(X, e_{1,2}) e_{1,2}) \Big) u \\
                &= \clm\Big(\Gauss_{N}(p(x))\ P_{\mathcal{H}}(X)\Big) u,
            \end{aligned}
        \end{equation*}
        and thereby analogously verifies \labelcref{eq:gauss-sharped-identity}.
    \end{proof}
    
    Set \(\mathcal{V} \coloneq \ker(\D p)\), so \(\T M = \mathcal{V} \oplus \mathcal{H}\).
    We need to show that both subbundles \(\mathcal{V}\) and \(\mathcal{H}\) are parallel to establish the desired product splitting.
    Here the problem we study is drastically simpler than that of \textcite{riedler-tony} because the map to the torus allows us to utilize the Ricci curvature formula directly in the Bochner method.

    \begin{proposition} \label{prop:vertical_parallel}
        The bundle \(\mathcal{V} = \ker(\D p)\) admits a global frame consisting of parallel vector fields on \(M\).
    \end{proposition}
    \begin{proof}
        For each \(i \in \{1, \dotsc, l\}\), let \(t_i = \proj_{\Sphere^1_{r_i}}^\ast (\D \theta) \in \HZ^1(N \times T; \R)\) be the de Rham cohomology class corresponding to the \(i\)-th circle factor in \(T = \Sphere^1_{r_1} \times \dotsm \times \Sphere^1_{r_l}\).
        Since \(f \colon M \to N \times T\) has non-zero degree, it follows that \(f^\ast(\vol_{N} \wedge t_1 \wedge \dotsc \wedge t_l) \neq 0 \in \HZ^{n+l}(M;\R)\).
        In particular, \(f^\ast t_1 \wedge \dotsc \wedge f^\ast t_l \neq 0 \in \HZ^{l}(M; \R)\).
        Let \(\tau_i \in \Omega^1(M)\) be the unique harmonic representative of the cohomology class \(f^\ast t_i \in \HZ^1(M;\R)\) for  \(i = 1, \dotsc, l\).
        In the case \(n \geq 3\), the Bochner formula and \labelcref{eq:ricci-identity} imply that
        \begin{align*}
        0 &= \int_{M} \abs{\nabla \tau_i}^2 + \Ric_{M}(\tau_i^\sharp, \tau_i^\sharp) \dV \\
         &= \int_{M} \abs{\nabla \tau_i}^2 + \Ric_{N}(\D p(\tau_i^\sharp), \D p(\tau_i^\sharp)) \dV \\
        &\geq \int_{M} \abs{\nabla \tau_i}^2 + C_N \abs{\D p(\tau_i^\sharp)}^2 \dV \geq 0,
        \end{align*}
        for some \(C_N > 0\) using that \(N\) has positive Ricci curvature.
        Therefore \(\nabla \tau_i = 0\) and \(\D p(\tau_i^\sharp) = 0\) for all \(i = 1, \dotsc, l\).
        In the case \(n = 2\), we obtain the same consequence by the analogous argument using the Bochner formula and \labelcref{eq:gauss-identity}.
        For \(n = 0\), we just have that \(M\) is Ricci flat, \(\mathcal{V} = \T M\), and the Bochner formula still shows that \(\nabla \tau_i = 0\) for all \(i = 1, \dotsc, l\).
        
        We conclude in all cases that \(\tau_1^\sharp, \dotsc, \tau_l^\sharp\) are parallel vector fields which lie pointwise in the subbundle \(\mathcal{V}\).
        Since \(\tau_1 \wedge \dotsc \wedge \tau_l\) represents a non-vanishing cohomology class, this means that \(\tau_1, \dotsc, \tau_l\) are linearly independent, and since they are parallel, they are even pointwise linearly independent.
        Therefore, for dimension reasons, \(\tau_1^\sharp, \dotsc, \tau_l^\sharp\) span the bundle \(\mathcal{V}\) and the proposition is proved.
    \end{proof}
    
    We are now ready to complete the proof of \cref{thm:general-product-splitting}.
    In the edge case \(n = 0\), \cref{prop:vertical_parallel} implies that \((M,g)\) admits a global parallel frame, so it is a flat torus and the desired statement is already proved.
    So let us assume \(n \geq 2\).
    First note that \(p = \proj_N \circ f \colon M \to N\) is a surjective submersion between closed manifolds, so it is a smooth fiber bundle by Ehresmann's theorem.
    Moreover, the horizontal bundle \(\mathcal{H} = \mathcal{V}^\perp\) is a parallel subbundle of \(\T M\) since the vertical bundle \(\mathcal{V}\) is.
    In particular, \(\mathcal{H}\) is integrable and thus defines a flat connection on the fiber bundle \(p \colon M \to N\).
    Since \(N\) is simply-connected, its monodromy is trivial and thus parallel transport yields a diffeomorphism \(\Psi \colon N \times F \xrightarrow{\cong} M\), where \(F = p^{-1}(y_0)\) is some fiber, and such that \(p \circ \Psi = \proj_N\), \(\mathcal{H} = \D \Psi(\T N \oplus 0)\) and \(\mathcal{V} = \D \Psi(0 \oplus \T F)\).
    Since \(\mathcal{V}\) and \(\mathcal{H}\) are parallel, the metric splits as a Riemannian product, that is, \(\Psi^\ast g = \proj_{N}^\ast g_{S} \oplus \proj_{F}^\ast g_{F}\), where \(g_S\) and \(g_F\) are metrics on \(N\) and \(F\), respectively.
    By \cref{prop:vertical_parallel}, each fiber \(F\) is totally geodesic and the induced metric \(g_F\) on \(F \subset M\) admits a global frame of parallel vector fields.
    Thus \((F,g_F)\) is a flat torus.
    By \cref{prop:submersion_package}, the map \(\proj_{T} \circ f \colon M \to T\) has vanishing derivative along the horizontal bundle.
    Thus \(\proj_T \circ f \circ \Psi \colon N \times F \to T\) is independent of the \(N\) coordinate and so we can write \(\proj_T \circ f \circ \Psi = h \circ \proj_F,\) for some smooth map \(h \colon F \to T\).
    Together, this means that \(f \circ \Psi = \id_{N} \times h\) and so \(0 \neq \deg(f) = \deg(h)\).
    
    We set \(\Phi \coloneq \Psi^{-1}\) and in the case \(n = 2\), we define \((S,g_S) \coloneq (N, g_S)\) and \(\phi \coloneq \id_{N}\).
    We now show that these choices satisfy all the desired properties.
    
    First consider again the case \(n \geq 3\).
    Then \(p\) is actually a Riemannian submersion which in our description translates to \(g_S = g_N\) and \(\Phi = \Psi^{-1} \colon (M, g) \to (N \times F, g_{N} \oplus g_{F})\) is a Riemannian isometry.
    
    In the case \(n = 2\), \(p\) is not necessarily a Riemannian submersion, so it does not follow that \(g_S = g_N\), but we have that \(\Lambda^2 \D p|_{\mathcal{H}} \colon \Lambda^2 \mathcal{H} \to \Lambda^2 \T N\) is a bundle isometry, so \(\id_{N} \colon (N, g_S) \to (N, g_{N})\) is area-density-preserving.
    Moreover, as a consequence of \labelcref{eq:scalar_curv_equal}, we conclude that the Gauß curvatures agree, i.e. that \(\Gauss_{(N,g_S)} = \Gauss_{(N,g_N)}\).
    
    To finish the proof, it only remains to verify that \(h\) is \(1\)-Lipschitz in either case.
    This follows because \(\phi \times h = f \circ \Psi \colon (S \times F, g_{S} \oplus g_{F}) \to (N \times T, g_{N} \oplus g_{T})\) is area non-increasing and \(\phi\) is area-preserving.
    Indeed, this means that there exists a \(g_S\)-unit vector \(\eta \in \T_{y} S\) such that \(\abs{\D \phi(\eta)}_{g_N} \geq 1\).
    If \(h\) were not \(1\)-Lipschitz, there would exist a \(g_F\)-unit vector \(\xi \in \T_{x} F\) such that \(\abs{(\D h)(\xi)}_{g_T} > 1\).
    Then 
    \[
        \abs{\D (\phi \times h)(\eta \wedge \xi)} = \abs{\D \phi(\eta) \wedge \D h(\xi)} = \abs{\D \phi(\eta)} \abs{\D h(\xi)} > 1,
    \]
    which is a contradiction to the area non-increasing property of \(\phi \times h\), and hence we may conclude the proof.
\end{proof}

\begin{remark} \label{rem:family-index-comparison}
    In light of the \(n = 0\) case, let us briefly compare this to the existing spinorial approaches to the Geroch conjecture.
    Gromov--Lawson's~\cite{Gromov-Lawson} method uses their notion of \emph{enlargeability} and one could paraphrase it as being a non-sharp version of Llarull's theorem applied to large covers.
    A more structural point of view uses the non-vanishing of the \emph{Rosenberg index}~\cite{Rosenberg:PSCNovikovI} derived from the strong Novikov conjecture for the group \(\Z^l\).
    However, since the group is commutative, this also reduces to a family index (originating in Lusztig’s~\cite{Lusztig} classical approach to the Novikov conjecture).
    In the Clifford-linear setting, this yields a K-theory class \(\alpha_{\Z^l}(\Torus^{l}) \in \K_{l}(\Cstar \Z^l)\) which is the image of the K-homological fundamental class \([\Dirac_{\Torus^l}] \in \K_{l}(\Torus^l)\) under the Novikov assembly map \(\nu \colon \K_{l}(\Torus^l) \xrightarrow{\cong} \K_{l}(\Cstar \Z^l)\).
    Identifying \(\Cstar \Z^l\) via Fourier transform with the continuous functions on the dual torus \(\hat{\Torus}^l = \Hom(\Z^l, \U(1))\), we can view \(\alpha_{\Z^l}(\Torus^{l}) \in \K^{-l}(\hat{\Torus}^l)\).
    More concretely, consider the family of flat complex line bundles \((\mathcal{F}_{\chi} \to \Torus^{l})_{\chi \in \hat{\Torus}^l}\), where each \(\mathcal{F}_{\chi} \to \Torus^l\) is the flat line bundle determined by the character \(\chi \in \hat{\Torus}^l\).
    Then the class \(\alpha_{\Z^l}(\Torus^{l}) \in \K^{-l}(\hat{\Torus}^l)\) is the family index of the \(\CCl_{l,0}\)-Dirac operator on \(\Torus^{l}\) twisted by the family \(\mathcal{F}_{\bullet}\).
    In contrast, the family index class we use in our construction naturally lives in a different formal K-theory degree, namely \(\K^{l}(\Torus^l)\).
    Moreover, since our class is equal to the Bott class, it also should not be expected to be identified with the Rosenberg index via some natural isomorphism.
    This is because \(\alpha_{\Z^l}(\Torus^{l}) = \nu([\Dirac_{\Torus^l}])\) retains the full topological information included in the K-homological fundamental class, whereas our class only sees the \enquote{top degree}.
    The difference would certainly become visible in the real setting (where \(\hat{\Torus}^l\) needs to be interpreted as a Real space in the sense of KR-theory) because of the presence of non-trivial degree \(0\) information coming from the \(\alpha\)-invariant.
\end{remark}

\printbibliography
\end{document}